\documentclass[12pt]{amsart}
\usepackage{}
\usepackage{amsmath}
\usepackage{amsfonts}
\usepackage{amssymb}
\usepackage[all,cmtip]{xy}           %xypic macro for latex2.09

\usepackage{bbding}
\usepackage{txfonts}
\usepackage[shortlabels]{enumitem}
\usepackage{ifpdf}
\ifpdf
\usepackage[colorlinks,final,backref=page,hyperindex]{hyperref}
\else
\usepackage[colorlinks,final,backref=page,hyperindex,hypertex]{hyperref}
\fi
\usepackage{tikz}
\usepackage[active]{srcltx}

\usepackage{difftrees}
\makeatletter

\newtheorem{theorem}{Theorem}[section]
\newtheorem{proposition}[theorem]{Proposition}
\newtheorem{lemma}[theorem]{Lemma}
\newtheorem{corollary}[theorem]{Corollary}
\newtheorem{prop-def}{Proposition-Definition}[section]

\theoremstyle{definition}
\newtheorem{definition}[theorem]{Definition}

\newtheorem{remark}[theorem]{Remark}
\newtheorem{example}[theorem]{Example}

\newcommand{\nc}{\newcommand}

\newcommand {\emptycomment}[1]{}

\nc{\delete}[1]{{}}
\nc{\mmargin}[1]{}

\nc{\mlabel}[1]{\label{#1}}  % Use this to suppress names
\nc{\mcite}[1]{\cite{#1}}  % Use this to suppress names
\nc{\mref}[1]{\ref{#1}}  % Use this to suppress names
\nc{\meqref}[1]{\eqref{#1}}  % Use this to suppress names
\nc{\mbibitem}[1]{\bibitem{#1}} % Use this to show number
\delete{
	\nc{\mlabel}[1]{\label{#1}  % Use the next two lines to show names
		{\hfill \hspace{1cm}{\bf{{\ }\hfill(#1)}}}}
	\nc{\mcite}[1]{\cite{#1}{{\bf{{\ }(#1)}}}}  % Use this lines to show names
	\nc{\mref}[1]{\ref{#1}{{\bf{{\ }(#1)}}}}  % Use this lines to show names
	\nc{\meqref}[1]{\eqref{#1}{{\bf{{\ }(#1)}}}}  % Use this lines to show names
	\nc{\mbibitem}[1]{\bibitem[\bf #1]{#1}} % Use this to show name
}

 \font\cyrs=wncyr7

\newcommand{\bk}{{\mathbf{k}}}
\nc{\vep}{\varepsilon}
\nc{\bin}[2]{ (_{\stackrel{\scs{#1}}{\scs{#2}}})}  %binomial coeff
\nc{\binc}[2]{(\!\! \begin{array}{c} \scs{#1}\\
		\scs{#2} \end{array}\!\!)}  %binomial coeff
\nc{\bincc}[2]{  ( {\scs{#1} \atop
		\vspace{-1cm}\scs{#2}} )}  %binomial coeff
\nc{\oline}[1]{\overline{#1}}
\nc{\mapm}[1]{\lfloor\!|{#1}|\!\rfloor}
\nc{\bs}{\bar{S}}
\nc{\cast}{{\,\mbox{\raisebox{.8pt}{$\scriptstyle \circledast$}}\,}}
\nc{\la}{\longrightarrow}
\nc{\ot}{\otimes}
\nc{\rar}{\rightarrow}
\nc{\dar}{\downarrow}
\nc{\dap}[1]{\downarrow \rlap{$\scriptstyle{#1}$}}
\nc{\defeq}{\stackrel{\rm def}{=}}
\nc{\dis}[1]{\displaystyle{#1}}
\nc{\dotcup}{\ \displaystyle{\bigcup^\bullet}\ }
\nc{\hcm}{\ \hat{,}\ }
\nc{\hts}{\hat{\otimes}}
\nc{\hcirc}{\hat{\circ}}
\nc{\lleft}{[}
\nc{\lright}{]}
\nc{\curlyl}{\left \{ \begin{array}{c} {} \\ {} \end{array}
	\right .  \!\!\!\!\!\!\!}
\nc{\curlyr}{ \!\!\!\!\!\!\!
	\left . \begin{array}{c} {} \\ {} \end{array}
	\right \} }
\nc{\longmid}{\left | \begin{array}{c} {} \\ {} \end{array}
	\right . \!\!\!\!\!\!\!}
\nc{\ora}[1]{\stackrel{#1}{\rar}}
\nc{\ola}[1]{\stackrel{#1}{\la}}%${\Bbb Z}$
\nc{\scs}[1]{\scriptstyle{#1}} \nc{\mrm}[1]{{\rm #1}}
\nc{\dirlim}{\displaystyle{\lim_{\longrightarrow}}\,}
\nc{\invlim}{\displaystyle{\lim_{\longleftarrow}}\,}
\nc{\dislim}[1]{\displaystyle{\lim_{#1}}} \nc{\colim}{\mrm{colim}}
\nc{\mvp}{\vspace{0.3cm}} \nc{\tk}{^{(k)}} \nc{\tp}{^\prime}
\nc{\ttp}{^{\prime\prime}} \nc{\svp}{\vspace{2cm}}
\nc{\vp}{\vspace{8cm}}
\nc{\modg}[1]{\!<\!\!{#1}\!\!>}
\nc{\intg}[1]{F_C(#1)}
\nc{\lmodg}{\!<\!\!}
\nc{\rmodg}{\!\!>\!}
\nc{\cpi}{\widehat{\Pi}}
\nc{\ssha}{{\mbox{\cyrs X}}} %sha as product
\nc{\tsha}{{\mbox{\cyrt X}}}
\nc{\shpr}{\diamond}    %Shuffle product
\nc{\labs}{\mid\!}
\nc{\rabs}{\!\mid}
\nc{\di}{\diamond}

\nc{\ad}{\mrm{ad}}
\nc{\rRB}{\mathsf{rRB}}
\nc{\cocrRB}{\mathsf{cocrRB}}
\nc{\PH}{\mathsf{PH}}
\nc{\cocPH}{\mathsf{cocPH}}
\nc{\ann}{\mrm{ann}}
\nc{\Aut}{\mrm{Aut}}
\nc{\Der}{\mrm{Der}}
\nc{\Sym}{\mrm{Sym}}
\nc{\br}{\mrm{bre}}
\nc{\can}{\mrm{can}}
\nc{\Cont}{\mrm{Cont}}
\nc{\rchar}{\mrm{char}}
\nc{\cok}{\mrm{coker}}
\nc{\de}{\mrm{dep}}
\nc{\dtf}{{R-{\rm tf}}}
\nc{\dtor}{{R-{\rm tor}}}

\nc{\Dif}{\mrm{Diff}}
\nc{\Div}{\mrm{Div}}
\nc{\End}{\mrm{End}}
\nc{\Ext}{\mrm{Ext}}
\nc{\Fil}{\mrm{Fil}}
\nc{\Fr}{\mrm{Fr}}
\nc{\Frob}{\mrm{Frob}}
\nc{\Gal}{\mrm{Gal}}
\nc{\GL}{\mrm{GL}}
\nc{\Gr}{\mrm{Gr}}
\nc{\Hom}{\mrm{Hom}}
\nc{\Hoch}{\mrm{Hoch}}
\nc{\hsr}{\mrm{H}}
\nc{\hpol}{\mrm{HP}}
\nc{\id}{\mrm{id}}
\nc{\im}{\mrm{im}}
\nc{\inv}{\mrm{inv}}
\nc{\Id}{\mrm{Id}}
\nc{\ID}{\mrm{ID}}
\nc{\Irr}{\mrm{Irr}}
\nc{\incl}{\mrm{incl}}
\nc{\length}{\mrm{length}}
\nc{\NLSW}{\mrm{NLSW}}
\nc{\Lie}{\mrm{Lie}}
\nc{\mchar}{\rm char}
\nc{\mpart}{\mrm{part}}
\nc{\ql}{{\QQ_\ell}}
\nc{\qp}{{\QQ_p}}
\nc{\rank}{\mrm{rank}}
\nc{\rcot}{\mrm{cot}}
\nc{\rdef}{\mrm{def}}
\nc{\rdiv}{{\rm div}}
\nc{\rtf}{{\rm tf}}
\nc{\rtor}{{\rm tor}}
\nc{\res}{\mrm{res}}
\nc{\SL}{\mrm{SL}}
\nc{\Spec}{\mrm{Spec}}
\nc{\tor}{\mrm{tor}}
\nc{\Tr}{\mrm{Tr}}
\nc{\tr}{\mrm{tr}}
\nc{\wt}{\mrm{wt}}

\nc{\bfk}{{\bf k}}
\nc{\bfone}{{\bf 1}}
\nc{\bfzero}{{\bf 0}}
\nc{\detail}{\marginpar{\bf More detail}
	\noindent{\bf Need more detail!}
	\svp}
\nc{\gap}{\marginpar{\bf Incomplete}\noindent{\bf Incomplete!!}
	\svp}
\nc{\FMod}{\mathbf{FMod}}
\nc{\Int}{\mathbf{Int}}
\nc{\Mon}{\mathbf{Mon}}
\nc{\remarks}{\noindent{\bf Remarks: }}
\nc{\Rep}{\mathbf{Rep}}
\nc{\Rings}{\mathbf{Rings}}
\nc{\Sets}{\mathbf{Sets}}
\nc{\Diff}{\mathbf{Diff}}
\nc{\Inte}{\mathbf{Inte}}
\nc{\U}{\mathbf{U}}

\nc{\BA}{{\mathbb A}}   \nc{\CC}{{\mathbb C}}
\nc{\DD}{{\mathbb D}}   \nc{\EE}{{\mathbb E}}
\nc{\FF}{{\mathbb F}}   \nc{\GG}{{\mathbb G}}
\nc{\HH}{{\mathbb H}}   \nc{\LL}{{\mathbb L}}
\nc{\NN}{{\mathbb N}}   \nc{\PP}{{\mathbb P}}
\nc{\QQ}{{\mathbb Q}}   \nc{\RR}{{\mathbb R}}
\nc{\TT}{{\mathbb T}}   \nc{\VV}{{\mathbb V}}
\nc{\ZZ}{{\mathbb Z}}   \nc{\TP}{\widetilde{P}}

\nc{\cala}{{\mathcal A}}    \nc{\calc}{{\mathcal C}}
\nc{\cald}{\mathcal{D}}     \nc{\cale}{{\mathcal E}}
\nc{\calf}{{\mathcal F}}    \nc{\calg}{{\mathcal G}}
\nc{\calh}{{\mathcal H}}    \nc{\cali}{{\mathcal I}}
\nc{\call}{{\mathcal L}}    \nc{\calm}{{\mathcal M}}
\nc{\caln}{{\mathcal N}}    \nc{\calo}{{\mathcal O}}
\nc{\calp}{{\mathcal P}}    \nc{\calr}{{\mathcal R}}

\nc{\cals}{{\mathcal S}}    \nc{\calt}{{\Omega}}
\nc{\calv}{{\mathcal V}}    \nc{\calw}{{\mathcal W}}
\nc{\calx}{{\mathcal X}}

\nc{\fraka}{{\mathfrak a}}
\nc{\frakb}{\mathfrak{b}}
\nc{\frakg}{{\frak g}}
\nc{\frakl}{{\frak l}}
\nc{\fraks}{{\frak s}}
\nc{\frakB}{{\frak B}}
\nc{\frakm}{{\frak m}}
\nc{\frakM}{{\frak M}}
\nc{\frakp}{{\frak p}}
\nc{\frakW}{{\frak W}}
\nc{\frakX}{{\frak X}}
\nc{\frakS}{{\frak S}}
\nc{\frakA}{{\frak A}}
\nc{\frakx}{{\frakx}}

\nc{\ynr}[1]{\textcolor{orange}{\underline{Yunnan:}#1 }}

\nc{\lir}[1]{\textcolor{red}{\underline{Li:}#1 }}

\delete{
	\input cyracc.def

	\newtheorem{theorem}{Theorem}[section]
	\newtheorem{lemma}[theorem]{Lemma}
	\newtheorem{proposition}[theorem]{Proposition}
	\newtheorem{corollary}[theorem]{Corollary}
	
	\theoremstyle{definition}
	\newtheorem{definition}[theorem]{Definition}
	\newtheorem{example}[theorem]{Example}

	\theoremstyle{remark}
	\newtheorem{remark}[theorem]{Remark}

	\allowdisplaybreaks
	
	\numberwithin{equation}{section}
	
}

\begin{document}

\title[ ]{Fundamental theorem of transposed Poisson $(A,H)$-Hopf modules}

\author{Yan Ning}
\address{School of Mathematics and Big Data, Jining University,
Qufu 273155, China}
\email{}

\author{Daowei Lu}
\address{School of Mathematics and Big Data, Jining University,
Qufu 273155, China}
\email{ludaowei620@126.com}

\author{Dingguo Wang}
\address{Department of General Education, Shandong Xiehe University, Jinan 250109, China}
\email{}

%\date{\today}

\begin{abstract}
Transposed Poisson algebra was introduced as a dual notion of the Poisson algebra by switching the roles played by the commutative associative operation and Lie operation in the Leibniz rule defining the Poisson algebra. Let $H$ be a Hopf algebra with a bijective antipode and $A$ an $H$-comodule transposed Poisson algebra. Assume that there exists an $H$-colinear map which is also an algebra map from $H$ to the transposed Poisson center of $A$. In this paper we generalize the fundamental theorem of $(A, H)$-Hopf modules to transposed Poisson $(A, H)$-Hopf modules and deduce relative projectivity in the category of transposed Poisson $(A, H)$-Hopf modules.
\end{abstract}

\keywords{Transposed Poisson algebra, Hopf algebra, Transposed Poisson Hopf module,  Fundamental theorem. \\ 
\qquad 2020 Mathematics Subject Classification. 17B63, 16T05.}

\maketitle

%\tableofcontents

\allowdisplaybreaks

\section*{Introduction} 

A Poisson algebra is a commutative associative unitary algebra $A$ endowed with a bilinear map $\{\cdot,\cdot\}:A\times A\rightarrow A$, providing $A$ with a Lie algebra structure and for all $a,a',a''\in A$ satisfying the {\bf Leibniz rule}
\begin{equation*}
\{a,a'a''\}=a'\{a,a''\}+\{a,a'\}a''.
\end{equation*}

Poisson algebras play a significant role in various branches of mathematics and physics. They are fundamental to integrable systems, Poisson geometry (including Poisson brackets and Poisson manifolds)\cite{BL,BV}. Additionally, they find applications in vertex algebras\cite{FB}, quantization theory\cite{CP,Hu}, deformation quantization\cite{Dz}, as well as in classical mechanics\cite{Ar}, Hamiltonian mechanics\cite{Od}, and quantum mechanics\cite{Di}. Furthermore, Poisson algebras are relevant in quantum field theory, Lie theory\cite{MR}, and operads\cite{GR}.

A transposed Poisson algebra\cite{BBGW} is a commutative associative unitary algebra $A$ endowed with a bilinear map $\{\cdot,\cdot\}:A\times A\rightarrow A$, such that $A$ with a Lie algebra structure and for all $a,a',a''\in A$ satisfying the {\bf transposed Leibniz rule}
\begin{equation*}
2a\{a',a''\}=\{a a',a''\}+\{a',a a''\}.
\end{equation*}

The difference of Poisson algebra and transposed Poisson algebra mainly lies in the fact that the roles played by the two binary operations $\cdot$ and $\{,\}$ in the Leibniz rule are switched.

While both the Poisson algebra and transposed Poisson algebra are built from a commutative associative operation and a Lie operation, the intersection of the compatibility conditions of both algebraic structures is trivial,  that is, the structures of transposed Poisson algebras and Poisson algebras are incompatible in this sense of \cite{BBGW}.

It is worthwhile to point out that transposed Poisson algebra has close relationship with other important algebraic structures. For example, taking the commutator in a Novikov-Poisson algebra or in a pre-Lie Poisson algebra gives rise to a transposed Poisson algebra; a derivation which is also a Lie derivation of a transposed Poisson algebra could produce a 3-Lie algebra; a transposed Poisson algebra with an involutive endomorphism of the commutative algebra which is at the same time an anti-endomorphism of the Lie algebra also gives rise to a 3-Lie algebra.

Let $H$ be a Hopf algebra, and $A$ a right $H$-comodule algebra. Doi \cite{D} established the fundamental theorem of relative Hopf modules, which states that if there exists a right $H$-comodule map $\phi:H\rightarrow A$ which is an algebra map, then for any relative right $(A,H)$-Hopf module $M$, the following isomorphism of relative right $(A,H)$-Hopf modules is obtained
$$A\ot_{A^{coH}}M\rightarrow M,\ a\ot m\mapsto a\cdot m.$$
This isomorphism plays an essential role in the theory of Hopf algebras, such as the integral theory, Nichols-Zoller theorem, Galois theory and so on \cite{Mon}.

Let $H$ be a Hopf algebra and $A$ a Poisson algebra such that $A$ is a $H$-comodule Poisson algebra. In~\cite{Gu} Gu$\acute{e}$d$\acute{e}$non established the fundamental theorem of Poisson $(A,H)$-Hopf modules, which generalized the fundamental isomorphism to the case of Poisson algebra. On account of the connections between Poisson algebra and transposed Poisson algebra, it is very natural and interesting to consider the fundamental theorem in the context of transposed Poisson algebra.

The paper is organized as follows. In section 1, we recall basic definitions on transposed Poisson algebra and introduce the notions of the center of transposed Poisson algebra, $H$-comodule transposed Poisson algebra and transposed Poisson Hopf modules. In section 2, we deduce the Maschke type theorem in the category of transposed Poisson Hopf modules, and show that under some conditions any transposed Poisson Hopf module which is injective as a Lie module is also an injective transposed Poisson Hopf modules. In section 3, we give the fundamental theorem of transposed Poisson $(A,H)$-Hopf modules and deduce relative projectivity  in the category of transposed Poisson $(A,H)$-Hopf modules.

\vspace{2mm}

\noindent
{\bf Convention.}
In this paper, let $\bk$ be a fixed field of characteristic 0. All the objects under discussion, including vector spaces, algebras and tensor products, are taken over $\bk$ otherwise specified.
For any coalgebra $(C,\Delta,\vep)$, we compress the Sweedler notation of the comultiplication $\Delta$ as $$\Delta(c)=c_1\otimes c_2$$ for simplicity.

\section{Preliminaries}
\def\theequation{1.\arabic{equation}}
\setcounter{equation} {0}

A Hopf algebra is an algebra $(H,m_H,1_H)$ and coalgebra $(H,\Delta_H,\varepsilon_H)$ that possesses an antipode $S_{H}:H\rightarrow H$ satisfying the defining relations
\begin{align*}
&(\Delta_{H}\otimes id_{H})\circ\Delta_{H}=(id_{H}\otimes\Delta_{H}) \circ\Delta_{H},\\
&m_{H}\circ[(S_{H}\otimes id_{H})\circ\Delta_{H}]=m_{H}\circ[(id_{H}\otimes S_{H})\circ\Delta_{H}]=\varepsilon_{H},\\
&(\varepsilon_{H}\otimes id_{H})\circ\Delta_{H}=(id_{H}\otimes\varepsilon_{H})\circ\Delta_{H}=id_{H}.
\end{align*}
For background on Hopf algebras and coactions of Hopf algebras on rings, we refer to \cite{Mon}. We will use Sweedler-Heyneman notation, writing:
$$\Delta_{H}(h)=h_{1}\otimes h_{2},\text{ for all }h\in H.$$

A right $H$-comodule is a vector space $M$ with structure map $\rho_{M}:M\rightarrow M\ot H,\ m\mapsto m_{(0)}\ot m_{(1)}$ satisfying
$$(id_M\ot \Delta_{H})\circ\rho_{M}=(\rho_{M}\ot id_N)\circ\rho_{M},\quad (id_M\ot \varepsilon_{H})\circ\rho_{M}=id_M.$$

The vector subspace
$$M^{coH}=\{m\in M|\rho_{M}(m)=m\otimes 1_{H}\}.$$
of $M$ is called the subspace of $H$-coinvariants of $M$.

An algebra $A$ is an $H$-comodule algebra if $A$ is an $H$-comodule satisfying
\begin{align*}
&(aa')_{(0)} \otimes (aa')_{(1)} = a_{(0)} a'_{(0)}\otimes a_{(1)} a'_{(1)}, \\ 
&(1_A)_{(0)}\otimes (1_A)_{(1)} = 1_A \otimes 1_H,
\end{align*}
for all $a,a'\in A$.

If $A$ is an $H$-comodule algebra, then the subspace $A^{coH}$ is a subalgebra of $A$.

\begin{definition}\cite{BBGW} 
A transposed Poisson algebra is a triple $(A,\cdot, \{, \})$, where $(A, \cdot)$ is a commutative associative algebra and $(A, \{, \})$ is a Lie algebra, satisfying the transposed Leibniz identity
\begin{equation}
2a\{b,c\}=\{ab,c\}+\{b,ac\},\label{1a}
\end{equation}
for all $a,b,c\in A$.
\end{definition}
A subspace $B$ of $A$ is called a transposed Poisson subalgebra of $A$ if $B$ is both a subalgebra and Lie subalgebra of $A$.

The transposed Poisson center $A^A$ of $A$ is given as follows
$$A^A=\{b\in A|b\{a,a'\}=\{ba,a'\}, \forall a,a'\in A\}.$$

\begin{lemma}\label{Lem:1a}
$A^A$ is a transposed Poisson subalgebra of $A$.
\end{lemma}

\begin{proof}
For all $b,b'\in A^A,a,a'\in A$, we have
$$\{ba,a'\}=b\{a,a'\}=-b\{a',a\}=-\{ba',a\}=\{a,ba'\},$$
and 
$$bb'\{a,a'\}=b\{b'a,a'\}=\{bb'a,a'\}.$$
Since $1_A\in A^A$, $A^A$ is a subalgebra of $A$. Moreover
\begin{align*}
2a\{b,b'\}&=\{ab,b'\}+\{b,ab'\}=b\{a,b'\}+\{b,a\}b'\\
&=\{a,bb'\}-\{a,b\}b'=\{a,bb'\}-\{a,b b'\}\\
&=0,
\end{align*}
therefore we obtain that $\{b,b'\}=0$ by setting $a=1_A$, that is, the Poisson bracket on $A^A$ is trivial. Thus $A^A$ is a transposed Poisson subalgebra of $A$.
\end{proof}

\begin{definition}
Let $H$ be a Hopf algebra and $A$ a transposed Poisson algebra. We call $A$ a right $H$-comodule transposed Poisson algebra if $A$ is a right $H$-comodule algebra and satisfies the following relation
\begin{equation}
\{a,a'\}_{(0)}\ot\{a,a'\}_{(1)}=\{a_{(0)},a'_{(0)}\}\ot a_{(1)}a'_{(1)},\label{1d}
\end{equation}
for all $a,a'\in A$.
\end{definition}

\begin{lemma}\label{Lem:1b}
Let $A$ be a right $H$-comodule transposed Poisson algebra. Then $A^{coH}$ is an $H$-subcomodule transposed Poisson algebra of $A$.
\end{lemma}

\begin{proof}
The proof is direct.
\end{proof}

If $A$ is a right $H$-comodule transposed Poisson algebra, set
$$A^{AcoH}=\{a\in A|a\in A^A\ \hbox{and}\ a\in A^{coH}\}.$$

\begin{definition}\cite{LS}
Let $A$ be a transposed Poisson algebra. A vector space $M$ is called a transposed Poisson $A$-module if $(M,\cdot)$ is an $A$-module and $(M,\di)$ a Lie $A$-module satisfying 
\begin{align}
&2\{a,b\}\cdot m=a\di(b\cdot m)-b\di(a\cdot m),\label{1b}\\
&2a\cdot(b\di m)=ab\di m+b\di(a\cdot m),\label{1c}
\end{align}
for all $a,b\in A,m\in M$.
\end{definition}

\begin{definition}
Let $M$ be a transposed Poisson $A$-module. Then the Lie action and algebra action are associative if for all $a\in A,b\in A^A,m\in M$, 
$$(ab)\di m=b\cdot(a\di m).$$
\end{definition}

\begin{remark}
Note that a transposed Poisson algebra $A$ is a transposed Poisson $A$-module with the actions given by the multiplications and Poisson bracket respectively. And the two actions are  associative.
\end{remark}

Let $M$ be a transposed Poisson $A$-module, define the Lie $A$-invariant elements of $M$ by
$$M^A=\{m\in M|\{a, a'\}\cdot m=a\di(a'\cdot m), \forall a,a'\in A\}.$$

\begin{lemma}\label{Lem:1c}
Let $M$ be a transposed Poisson $A$-module, then for all $b\in A^A,m\in M^A$, we have $b\di m=0$.
\end{lemma}

\begin{proof}
In the identity (\ref{1b}), let $a=1_A$, we get $1_A\di (b\cdot m)=b\di m$ by Lemma \ref{Lem:1a}, which implies that $b\di m=\{1_A,b\}\cdot m=0$.
\end{proof}

\begin{definition}\cite{Gu}
Let $A$ be an $H$-comodule algebra. A vector space $M$ is an $(A, H)$-Hopf module if $M$ is a left $A$ -module and a right $H$-comodule such that
$$(a\cdot m)_{(0)}\ot (a\cdot m)_{(1)}=a_{(0)}\cdot m_{(0)}\ot a_{(1)}m_{(1)},$$
for all $a\in A,m\in M$.
\end{definition}

\begin{definition}
Let $A$ be a right $H$-comodule transposed Poisson algebra. A vector space $M$ is called a transposed Poisson $(A,H)$-Hopf module if $M$ is a transposed Poisson $A$-module and an $(A,H)$-Hopf module such that for all $a\in A,m\in M$,
\begin{equation}
(a\di m)_{(0)}\ot (a\di m)_{(1)}=a_{(0)}\di m_{(0)}\ot a_{(1)}m_{(1)}.\label{1e}
\end{equation}
\end{definition}

A transposed Poisson $(A, H)$-Hopf module homomorphism between two transposed Poisson $(A, H)$-Hopf modules is an $A$-linear map, a Lie $A$-linear map and an $H$-colinear map. We denote by $_{\mathcal{P}A}\mathcal{M}^H$ the category of transposed Poisson $(A, H)$-Hopf modules with their homomorphisms. For $M,N\in\! _{\mathcal{P}A}\mathcal{M}^H$, we denote by $_{\mathcal{P}A}$Hom$^H(M,N)$ the vector space of transposed Poisson $(A, H)$-Hopf module homomorphisms from $M$ to $N$.

\begin{lemma}\label{Lem:1d}
Let $H$ be a Hopf algebra, $A$ a right $H$-comodule transposed Poisson algebra, and $M$ a transposed Poisson $(A,H)$-Hopf module. Then
\begin{itemize}
  \item [(1)] $M^A$ is an $H$-subcomodule of $M$.
  \item [(2)] $A^A$ is an $H$-subcomodule transposed Poisson algebra of $A$, and $A^{AcoH}$ is a transposed Poisson subalgebra of $A^A$.
  \item [(3)] $M^{AcoH}$ is a transposed Poisson $A^{AcoH}$-submodule of $M$.
\end{itemize}
\end{lemma}

\begin{proof}
(1) For all $a,a'\in A,m\in M^A,$ on one hand
\begin{align*}
&(1\ot S^{-1}(a_{(1)}a'_{(1)}))[(\{a_{(0)},a'_{(0)}\}\cdot m)_{(0)}\ot(\{a_{(0)},a'_{(0)}\}\cdot m)_{(1)}]\\
&=(1\ot S^{-1}(a_{(1)}a'_{(1)}))[\{a_{(0)},a'_{(0)}\}_{(0)}\cdot m_{(0)}\ot\{a_{(0)},a'_{(0)}\}_{(1)}m_{(1)}]\\
%&=(1\ot S^{-1}(a_{(1)}a'_{(1)}))[\{a_{(0)(0)},a'_{(0)(0)}\}\cdot m_{(0)}\ot a_{(0)(1)}a'_{(0)(1)}m_{(1)}]\\
&=(1\ot S^{-1}(a_{(1)2}a'_{(1)2}))[\{a_{(0)},a'_{(0)}\}\cdot m_{(0)}\ot a_{(1)1}a'_{(1)1}m_{(1)}]\\
&=\{a,a'\}\cdot m_{(0)}\ot m_{(1)},
\end{align*}
and on the other hand
\begin{align*}
&(1\ot S^{-1}(a_{(1)}a'_{(1)}))[(\{a_{(0)},a'_{(0)}\}\cdot m)_{(0)}\ot(\{a_{(0)},a'_{(0)}\}\cdot m)_{(1)}]\\
&=(1\ot S^{-1}(a_{(1)}a'_{(1)}))[(a_{(0)}\di(a'_{(0)}\cdot m))_{(0)}\ot(a_{(0)}\di(a'_{(0)}\cdot m))_{(1)}]\\
&=(1\ot S^{-1}(a_{(1)2}a'_{(1)2}))[(a_{(0)}\di(a'_{(0)}\cdot m_{(0)}))\ot a_{(1)1}a'_{(1)1} m_{(1)}]\\
&=(a\di(a'\cdot m_{(0)}))\ot m_{(1)},
\end{align*}
hence
$$(\{a,a'\}\cdot m_{(0)}-a\di(a'\cdot m_{(0)}))\ot m_{(1)}=0.$$ 

By a similar argument as in \cite{Gu}, we have $\{a,a'\}\cdot m_{(0)}-a\di(a'\cdot m_{(0)})=0$ for each summand $m_{(0)}$.  So $M^A$ is an $H$-subcomodule of $M$.

(2) By (1) $A^A$ is an $H$-subcomodule of $A$. Since $A^A$ is also a subalgebra and Lie subalgebra of $A$, $A^A$ is an $H$-subcomodule transposed Poisson algebra of $A$. Similar argument implies that $A^{AcoH}$ is a transposed Poisson subalgebra of $A^A$.

(3) Firstly $M$ is a transposed Poisson $A^{AcoH}$-module of $M$ since $A^{AcoH}$ is a transposed Poisson subalgebra of $A$. For all $a,a'\in A,b\in A^{AcoH},m\in M^{AcoH}$, we have 
$$\{a,a'\}\cdot(b\cdot m)=\{a,a'\}b\cdot m=\{a,a'b\}\cdot m=a\di(a'\cdot(b\cdot m)),$$
that is, $M^{AcoH}$ is a $A^{AcoH}$-submodule of $M$. And $M^{AcoH}$ is a trivial Lie $A^{AcoH}$-submodule of $M$ by Lemma \ref{Lem:1c}. The proof is completed.
\end{proof}

\vspace{2mm}

\section{Maschke type theorem in the category $_{\mathcal{P}A}\mathcal{M}^H$}
\def\theequation{2.\arabic{equation}}
\setcounter{equation} {0}

In what follows, we will say that a vector space $M$ is an $(\underline{A},H)$-comodule if $M$ is a Lie $A$-module, a right $H$-comodule and the relation (\ref{1e}) is satisfied. We denote by $_{\underline{A}}\mathcal{M}$ the category of Lie $A$-modules with Lie $A$-linear maps, and by $_{\underline{A}}\mathcal{M}^H$ the category of $(\underline{A},H)$-comodule with Lie $A$-linear and $H$-colinear maps. For $M,N\in\! _{\underline{A}}\mathcal{M}^H$, we denote by $_{\underline{A}}$Hom$^H(M,N)$ the vector space of Lie $A$-linear and $H$-colinear maps from $M$ to $N$.

\begin{lemma}\label{Lem:2a}
\begin{itemize}
  \item [(1)] Let $N$ be a Lie $A$-module. Then $N\ot H$ is an $(\underline{A},H)$-comodule with the Lie $A$-action and $H$-coaction given by
\begin{align*}
&a\di (n\ot h)=a_{(0)}\di n\ot a_{(1)}h,\\
&(n\ot h)_{(0)}\ot(n\ot h)_{(1)}=n\ot h_1\ot h_2,
\end{align*}
for all $a\in A, n\ot h\in N\ot H$.
  \item [(2)] Furthermore, if $H$ is commutative and $N$ is a transposed Poisson $A$-module, then $N\ot H$ is a transposed Poisson $(A, H)$-Hopf module with the $A$-action given by
$$a\cdot (n\ot h)= a_{(0)}\cdot n\ot a_{(1)}h,$$
for all $a\in A, n\ot h\in N\ot H$.
\end{itemize}
\end{lemma}

\begin{proof}
In \cite{Gu}, the conclusion (1) have already been verified. We only need to prove (2). For all $a,a'\in A, n\ot h\in N\ot H$,
\begin{align*}
2\{a,a'\}\cdot (n\ot h)&=2\{a,a'\}_{(0)}\cdot n\ot\{a,a'\}_{(1)} h\\
&=2\{a_{(0)},a'_{(0)}\}\cdot n\ot a_{(1)}a'_{(1)} h\\
&=[a_{(0)}\di(a'_{(0)}\cdot n)-a'_{(0)}\di(a_{(0)}\cdot n)]\ot a_{(1)}a'_{(1)} h\\
&=(a_{(0)}\di(a'_{(0)}\cdot n))\ot a_{(1)}a'_{(1)} h-(a'_{(0)}\di(a_{(0)}\cdot n))\ot a_{(1)}a'_{(1)} h\\
&=a\di(a'\cdot (n\ot  h))-a'_{(0)}\di(a\cdot (n\ot h)),
\end{align*}
and
\begin{align*}
2a\cdot (a'\di(n\ot h))&=2a\cdot (a'_{(0)}\di n\ot a'_{(1)}h)\\
&=2a_{(0)}\cdot(a'_{(0)}\di n)\ot a_{(1)}a'_{(1)}h\\
&=[a_{(0)}a'_{(0)}\di n+a'_{(0)}\di(a_{(0)}\cdot n)]\ot a_{(1)}a'_{(1)}h\\
&=(a_{(0)}a'_{(0)}\di n)\ot a_{(1)}a'_{(1)}h+(a'_{(0)}\di(a_{(0)}\cdot n))\ot a_{(1)}a'_{(1)}h\\
&=aa'\di(n\ot h)+a'\di(a\cdot(n\ot h)).
\end{align*}
The relation (\ref{1e}) obviously holds. The proof is completed.
\end{proof}

\begin{lemma}\label{Lem:2b}
\begin{itemize}
  \item [(1)] Let $M$ be an $(\underline{A},H)$-comodule and $N$ a Lie $A$-module. There exists a linear isomorphism
$$\gamma:\! _{\underline{A}}\hbox{Hom}^H(M,N\ot H)\rightarrow \! _{\underline{A}}\hbox{Hom}(M,N),\ f\mapsto(id\ot\varepsilon)\circ f.$$
The inverse $\gamma'$ of $\gamma$ is given by $\gamma'(g)=(g\ot id_H)\circ\rho_M$.
  \item [(2)] Let $H$ be commutative, $M$ a transposed Poisson $(A, H)$-Hopf module and $N$ a transposed Poisson $A$-module. There exists a linear isomorphism
$$\gamma:\! _{\mathcal{P}A}\hbox{Hom}^H(M,N\ot H)\rightarrow\! _{\mathcal{P}A}\hbox{Hom}(M,N),\ f\mapsto(id\ot\varepsilon)\circ f.$$
\end{itemize}
\end{lemma}

\begin{proof}
(1) The verification of $\gamma$ being an isomorphism is the same as in \cite{Gu}. 
 
(2) By Lemma \ref{Lem:2a}, $N\ot H$ is a transposed Poisson $(A, H)$-Hopf module. One could easily check that $\gamma(f)$ and $\gamma'(g)$ are $A$-linear for all $f\in\! _{\mathcal{P}A}\hbox{Hom}^H(M,N\ot H)$ and $g\in\! _{\mathcal{P}A}\hbox{Hom}(M,N)$.
\end{proof}

By Lemma \ref{Lem:2b}, we immediately have

\begin{corollary}\label{cor:2c}
\begin{itemize}
  \item [(1)] If $N$ is an injective Lie $A$-module, then $N\ot H$ is an injective $(A, H)$-comodule.
  \item [(2)] Let $H$ be commutative. If $N$ is an injective transposed Poisson $A$-module, then $N\ot H$ is an injective transposed Poisson $(A, H)$-Hopf module.
\end{itemize}
\end{corollary}

\begin{theorem}\label{thm:2d}
Let $A$ be a $H$-comodule transposed Poisson algebra, and $M$ a transposed Poisson $(A, H)$-Hopf module such that the Lie action and algebra action are associative. Assume that there exists an $H$-colinear map $\phi:H\rightarrow A^A$ such that $\phi(1_H)=1_A$.
\begin{itemize}
  \item [(1)]  If $H$ is commutative or $\phi$ is an algebra map, then every transposed Poisson $(A, H)$-Hopf module which is injective as a Lie $A$-module is an injective $(\underline{A}, H)$-comodule.
  \item [(2)] If $H$ is commutative, then every transposed Poisson $(A, H)$-Hopf module which is injective transposed Poisson $A$-module is an injective transposed Poisson $(A, H)$-Hopf module.
\end{itemize}
\end{theorem}

\begin{proof}
(1) Let $M$ be a transposed Poisson $(A, H)$-Hopf module and consider the linear map $\lambda:M\ot H\rightarrow M$ defined by
$$\lambda(m\ot h)=\phi(hS^{-1}(m_{(1)}))\cdot m_{(0)},\quad \forall m\in M,h\in H.$$
Just as shown in \cite{Gu}, $\lambda\circ\rho_M=id_M$ and $\lambda$ is $H$-colinear, which means that $\rho_M$ is an injective map. 

For all $a\in A,m\in M,h\in H$, we have
\begin{align*}
\lambda(a\di (m\ot h))&=\lambda(a_{(0)}\di m\ot a_{(1)}h)\\
&=\phi(a_{(1)}hS^{-1}((a_{(0)}\di m)_{(1)}))\cdot (a_{(0)}\di m)_{(0)}\\
&=\phi(a_{(1)2}hS^{-1}(a_{(1)1}m_{(1)}))\cdot (a_{(0)}\di m_{(0)}),
\end{align*}
and
\begin{align*}
a\di\lambda(m\ot h)&=a\di(\phi(hS^{-1}_H(m_{(1)}))\cdot m_{(0)})\\
&=2\phi(hS^{-1}_H(m_{(1)}))\cdot(a\di m)-(a\phi(hS^{-1}_H(m_{(1)})))\di m\\
&=\phi(hS^{-1}_H(m_{(1)}))\cdot(a\di m),
\end{align*}
where the last identity holds since $\phi(hS^{-1}_H(m_{(1)}))\in A^A$ and the Lie action and algebra action are associative. Now if $H$ is commutative, we obtain
$$\lambda(a\di (m\ot h))=a\di\lambda(m\ot h),$$
that is, $\lambda$ is Lie $A$-linear and a homomorphism in $_{\underline{A}}\mathcal{M}^H$. If $\phi$ is an algebra map, using the fact that $A$ is commutative, we have
\begin{align*}
&\phi(a_{(1)2}hS^{-1}(a_{(1)1}m_{(1)}))\cdot (a_{(0)}\di m_{(0)})\\
&=\phi(a_{(1)2}S^{-1}(a_{(1)1}))\phi(hS^{-1}(m_{(1)}))\cdot (a_{(0)}\di m_{(0)})\\
&=\phi(hS^{-1}(m_{(1)}))\cdot (a\di m_{(0)}).
\end{align*}
It follows that $\lambda$ is Lie $A$-linear and a homomorphism in $_{\underline{A}}\mathcal{M}^H$. It is easy to verify that $\rho_M$ is a homomorphism of $(\underline{A},H)$-comodule. If $M$ is an injective Lie $A$-module, by a similar argument as in \cite{Gu}, $M$ is an injective $(\underline{A},H)$-comodule.

(2) By a similar computation as above, we obtain that $\lambda$ given in (1) is $A$-linear. Thus $\lambda$ is a homomorphism of transposed Poisson $(A, H)$-Hopf modules. By the same arguments as \cite{Gu}, we conclude that an injective transposed Poisson $A$-module is also an injective transposed Poisson $(A,H)$-Hopf module.
\end{proof}

\begin{corollary}
Let $A$ be a commutative $H$-comodule algebra, and $M$ an $(A, H)$-Hopf module. Assume that there is an $H$-colinear map $\phi:H\rightarrow A$ such that $\phi(1_H)=1_A$.
\begin{itemize}
  \item [(1)] Then every $(A, H)$-Hopf module is an injective $H$-comodule.
  \item [(2)] Let $H$ be commutative. Then every $(A, H)$-Hopf module which is injective as an $A$-module is an injective $(A, H)$-Hopf module.
\end{itemize}
\end{corollary}

\begin{proof}
(1) Since any commutative algebra $A$ is a transposed Poisson algebra with a trivial Poisson bracket, we have $A=A^A$, and any commutative $H$-comodule algebra $A$ is an $H$-comodule transposed Poisson algebra. Any $H$-comodule is an $(\underline{A}, H)$-comodule with a trivial Lie $A$-action, and any $(A, H)$-Hopf module is a transposed Poisson $(A, H)$-Hopf module with a trivial Lie $A$-action. So the map $\lambda$ in the proof of Theorem \ref{thm:2d} is Lie $A$-linear. Also any homomorphism of $H$-comodules is a homomorphism of $(\underline{A}, H)$-comodules. The result follows from Theorem \ref{thm:2d}(1).

(2) Since any homomorphism of $(A, H)$-Hopf modules is a homomorphism of transposed Poisson $(A, H)$-Hopf modules, the result follows from Theorem \ref{thm:2d}(2).
\end{proof}

\vspace{2mm}

\section{Fundamental theorem of transposed Poisson Hopf modules}
\def\theequation{3.\arabic{equation}}
\setcounter{equation} {0}

In what follows, we will denote $B=A^{AcoH}$.
\begin{lemma}\label{Lem:3a}
Let $M$ be a transposed Poisson $B$-module. Then $A\ot_BM$ is a transposed Poisson $(A,H)$-Hopf module, where the $A$-action and the Lie $A$-action are given by
\begin{align*}
&a'\cdot (a\ot_B m)=a' a\ot_Bm\\
&a'\di (a\ot_B m)=a'\di a\ot_Bm,
\end{align*}
and the $H$-coaction is given by
$$\rho_{A\ot_BM}(a\ot_B m)=a_{(0)}\ot_Bm\ot a_{(1)},$$
for all $a,a'\in A,m\in M$.
\end{lemma}

\begin{proof}
It is obvious that the coaction is well defined. For all $a,a'\in A,m\in M$, it is straightforward to see that 
$$(a'\cdot(a\ot_B m))_{(0)}\ot (a'\cdot(a\ot_B m))_{(1)}=a'_{(0)}\cdot(a\ot_Bm)_{(0)}\ot a'_{(1)}(a\ot_Bm)_{(1)},$$
hence $A\ot_BM$ is an $(A,H)$-Hopf module. For $b\in B$,
\begin{align*}
a'\di(ab\ot_Bm)&=(a'\di(ab))\ot_Bm\\
&\stackrel{(\ref{1a})}{=}(2b\{a',a\}-\{ba',a\})\ot_Bm\\
&=b\{a',a\}\ot_Bm+b\{a',a\}\ot_Bm-\{ba',a\}\ot_Bm\\
&=\{a',a\}\ot_Bb\cdot m+(b\{a',a\}-\{ba',a\})\ot_Bm\\
&=\{a',a\}\ot_Bb\cdot m\\
&=a'\di (a\ot_Bb\cdot m),
\end{align*}
where the fifth identity uses the definition of $A^A$, therefore the Lie action is well defined and $A\ot_BM$ is a Lie $A$-module. Since
\begin{align*}
2\{a'',a'\}\cdot (a\ot_Bm)&=(2\{a'',a'\}a)\ot_Bm\\
&=\{aa'',a'\}\ot_Bm+\{a'',aa'\}\ot_Bm\\
&=(a''\di a'a)\ot_Bm-(a'\di a''a)\ot_Bm\\
&=a''\di(a'a\ot_Bm)-a'\di(a''a\ot_Bm)\\
&=a''\di(a'\cdot(a\ot_Bm))-a'\di(a''\cdot(a\ot_Bm)),
\end{align*}
and
\begin{align*}
2a''\cdot (a'\di(a\ot_Bm))&=(2a''\{a', a\})\ot_Bm\\
&=\{a''a',a\}\ot_Bm+\{a',a''a\}\ot_Bm\\
&=a''a'\di(a\ot_Bm)+a'\di(a''\cdot(a\ot_Bm)),
\end{align*}
therefore $A\ot_BM$ is a  transposed Poisson $A$-module. Finally
\begin{align*}
&(a'\di(a\ot_B m))_{(0)}\ot (a'\di(a\ot_B m))_{(1)}\\
&=(\{a',a\}\ot_B m)_{(0)}\ot (\{a',a\}\ot_B m)_{(1)}\\
&=\{a',a\}_{(0)}\ot_B m\ot \{a',a\}_{(1)}\\
&=\{a'_{(0)},a_{(0)}\}\ot_B m\ot a'_{(1)}a_{(1)}\\
&=(a'_{(0)}\di(a_{(0)}\ot_B m))\ot a'_{(1)}a_{(1)}\\
&=(a'_{(0)}\di(a\ot_B m)_{(0)})\ot a'_{(1)}(a\ot_B m)_{(1)},
\end{align*}
which means that $A\ot_BM$ is a transposed Poisson $(A,H)$-Hopf module. The proof is completed.
\end{proof}

For every transposed Poisson $(A,H)$-Hopf module $M$, from Lemma \ref{Lem:1c} and Lemma \ref{Lem:3a}, we deduce that $A\ot_BM^{AcoH}$ is a Poisson $(A,H)$-Hopf module.

\begin{lemma}\label{Lem:3b}
Let $M$ be a transposed Poisson $(A,H)$-Hopf module. Then the linear map 
$$\alpha:A\ot_BM^{AcoH}\rightarrow M,\ a\ot_Bm\mapsto a\cdot m$$ 
is a morphism of transposed Poisson $(A,H)$-Hopf modules.
\end{lemma}

\begin{proof}
Clearly $\alpha$ is well defined and a left $A$-module map. For all $a',a\in A,m\in M$,
\begin{align*}
\alpha(a'\di(a\ot_B m))&=\alpha(\{a',a\}\ot_B m)\\
&=\{a',a\}\cdot m\\
&=a'\di(a\cdot m)\\
&=a'\di\alpha(a\ot_B m),
\end{align*}
thus $\alpha$ is a Lie $A$-module map. It is easy to see that $\alpha$ is right $H$-colinear. The proof is completed. 
\end{proof}

Consider $H$ as a right $H$-comodule via $\Delta_H$ and assume that there exists a homomorphism of right $H$-comodule $\phi:H\rightarrow A$ such that $\phi(1_H)=1_A$. For any transposed Poisson $(A,H)$-Hopf module $M$, define linear map
$$p_M:M\rightarrow M,\quad m\mapsto\phi(S^{-1}(m_{(1)}))\cdot m_{(0)}.$$

\begin{lemma}\cite{Gu}
With the above notations, we have
\begin{itemize}
  \item [(1)] Im$p_M=M^{coH}$.
  \item [(2)] $p^2_M=p_M$, and $p_M$ is a projection.
\end{itemize}
\end{lemma}

%For transposed Poisson $(A,H)$-Hopf module $M$ and $a\in A,m\in M$, set
%%$$a\di' m=p_M(\{a,1_A\}\cdot m).$$
%$$a\di' p_M(m)=\{a,1_A\}\cdot p_M(m).$$

\begin{lemma}\label{Lem:3c}
Assume that there exists a homomorphism of right $H$-comodule $\phi:H\rightarrow A^A$ which is also an algebra map. If for all $a,a'\in A,m\in M$, 
\begin{equation}
\{1,aa'\}\cdot p_M(m)=\phi(a'_{(1)})\di(a\cdot p_M(a'_{(0)}\cdot m)),\label{3a}
\end{equation}
then $M^{coH}=M^{AcoH}$. Meanwhile we have
$$\{1,a\}\cdot p_M(m)=0.$$
\end{lemma}

\begin{proof}
For all $a,a'\in A,m\in M$, on one hand
$$\{a,a'\}\cdot p_M(m)=2a'\{a,1\}\cdot p_M(m)+\{1,aa'\}\cdot p_M(m),$$
and on the other hand
\begin{align*}
a\di(a'\cdot p_M(m))&=a\di(a'\phi(S^{-1}(m_{(1)}))\cdot m_{(0)})\\
&=a\di(\phi(S^{-1}(m_{(1)}))\varepsilon(a'_{(1)})\cdot (a'_{(0)}\cdot m_{(0)}))\\
&=a\di(\phi(a'_{(1)})\cdot p_M(a'_{(0)}\cdot m))\\
&=2\{a,\phi(a'_{(1)})\}\cdot p_M(a'_{(0)}\cdot m)+\phi(a'_{(1)})\di(a\cdot p_M(a'_{(0)}\cdot m))\\
&=2\{a,1\}\phi(a'_{(1)})\cdot p_M(a'_{(0)}\cdot m)+\phi(a'_{(1)})\di(a\cdot p_M(a'_{(0)}\cdot m))\\
&=2\{a,1\}a'\cdot  p_M(m)+\phi(a'_{(1)})\di(a\cdot p_M(a'_{(0)}\cdot m)).
\end{align*}
From the identity (\ref{3a}), we have $\{a,a'\}\cdot p_M(m)=a\di(a'\cdot p_M(m))$, namely, $M^{coH}\subseteq M^{AcoH}$ and $M^{coH}=M^{AcoH}$.
\end{proof}

\begin{remark}
For all $a,a',a''\in A$, if 
\begin{equation}
\{1,aa'\}\cdot p_A(a'')=\phi(a'_{(1)})\di(a\cdot p_A(a'_{(0)}\cdot a'')),\label{3b}
\end{equation}
then by Lemma \ref{Lem:3c}, $A^{coH}=A^{AcoH}$.
\end{remark}

\begin{theorem}\label{Thm:3d}
Let $M$ be a transposed Poisson $(A,H)$-Hopf module. Assume that  there exists a homomorphism of right $H$-comodule $\phi:H\rightarrow A^A$ which is also an algebra map such that the identities (\ref{3a}) and (\ref{3b}) hold. Then the linear map $\alpha$ given in Lemma \ref{Lem:3b} is an isomorphism of transposed Poisson $(A,H)$-Hopf modules.
\end{theorem}

\begin{proof}
Since $M^{coH}=M^{AcoH}$, we get a well defined linear map
$$\beta:M\rightarrow A\ot_BM^{AcoH},\quad m\mapsto \phi(m_{(1)})\ot_Bp_M(m).$$
By \cite{Gu} we could see that $\alpha\circ\beta=id_{M}$ and $\beta\circ\alpha=id_{A\ot_BM^{AcoH}}$. Therefore $\alpha$ is an isomorphism.
\end{proof}

By Theorem \ref{Thm:3d}, we could obtain the fundamental theorem of Doi but with the algebra commutative.

\begin{corollary}
Let $A$ be a commutative $H$-comodule algebra, and $M$ an $(A,H)$-Hopf module. Assume that there exists a right $H$-colinear map $\phi:H\rightarrow A$ which is also an algebra map. Let $B=A^{coH}$, then the linear map $\alpha:A\ot_BM^{coH}\rightarrow M,\ a\ot_B m\mapsto a\cdot m$ is an isomorphism of $(A,H)$-Hopf modules. 
\end{corollary}

\begin{proof}
Any commutative algebra is a transposed Poisson algebra with a trivial Poisson bracket, and $A=A^A$. It is clear that any commutative $H$-comodule algebra $A$ is an $H$-comodule transposed Poisson algebra. Any $(A,H)$-Hopf module $M$ is a transposed Poisson $(A,H)$-Hopf module with a trivial Lie $A$-action. Thus any homomorphism of $(A,H)$-Hopf modules is a homomorphism of transposed Poisson $(A,H)$-Hopf modules. Since the Lie $A$-action on $M$ is trivial, we have the relation (\ref{3a}).The result follows.
\end{proof}

Let $A$ be an $H$-comodule transposed Poisson algebra. Let $I$ be a vector subspace of $A$. We say that $I$ is

\begin{itemize}
    \item[(i)] an $H$-ideal of $A$ if $I$ is an ideal of $A$ and an $H$-subcomodule of $A$;
    \item[(ii)] a Poisson ideal of $A$ if $I$ is an ideal of $A$ and a Lie $A$-submodule of $A$;
    \item[(iii)] a Poisson $H$-ideal of $A$ if $I$ is an $H$-ideal of $A$ and a Poisson ideal of $A$, that is, an ideal of $A$, an $H$-subcomodule of $A$ and a Lie $A$-submodule of $A$.
\end{itemize}

We say that an $H$-comodule transposed Poisson algebra $A$ is Poisson $H$-simple if the only Poisson $H$-ideals of $A$ are 0 and $A$.

\begin{lemma}\label{Lem:3e}
Let $A$ be a Poisson $H$-simple $H$-comodule transposed Poisson algebra. Then $A^{AcoH}$ is a field.
\end{lemma}

\begin{proof}
Let $b\in A^{AcoH}, b\neq 0$, then $Ab$ is an ideal of $A$ since $A$ is commutative. For all $a,a'\in A$, we have
$$a'\di(ab)=\{a',ab\}=2\{a',a\}b+\{a,a'b\}=2\{a',a\}b+\{a,a'\}b=\{a',a\}b\in Ab,$$
that is, $Ab$ is a Lie $A$-submodule of $A$. It is clear that $Aa$ is an $H$-subcomodule of $A$. Therefore $Aa$ is a Poisson-idea of $A$. Since $Aa\neq0$ and $A$ is Poisson $H$-simple, we obtain $Aa=A$, so there exists $a'\in A$ such that $aa'=1$, namely, $a$ is invertible.
\end{proof}

By Theorem \ref{Thm:3d} and Lemma \ref{Lem:3e}, we immediately obtain the following corollary.

\begin{corollary}
Let $A$ be a Poisson $H$-simple $H$-comodule transposed Poisson algebra and $M$ a transposed Poisson $(A, H)$-Hopf module. Assume that there exists a right $H$-colinear map $\phi:H\rightarrow A^A$ which is also an algebra map, and that the identities (\ref{3a}) and (\ref{3b}) hold. Then $M$ is a free as an $A$-module with rank the dimension of $M^{AcoH}$ over $A^{AcoH}$.
\end{corollary}

Let $M,N\in\! _{\mathcal{P}A}\mathcal{M}^H$ and the morphism $f:M\rightarrow N$. For $m\in M^{AcoH},a,a'\in A$, we compute
\begin{align*}
&f(m)_{(0)}\ot f(m)_{(1)}=f(m_{(0)})\ot m_{(1)}=f(m)\ot 1_H,\\
&\{a,a'\}\cdot f(m)=f(\{a,a'\}\cdot m)=f(a\di(a'\cdot m))=a\di f(a'\cdot m)=a\di (a'\cdot f(m)),
\end{align*}
thus $f(m)\in N^{AcoH}$ and $f(M^{AcoH})\subseteq N^{AcoH}$. This gives rise to a functor 
$$G:\!  _{\mathcal{P}A}\mathcal{M}^H\rightarrow \! _B\mathcal{M},\quad M\mapsto M^{AcoH}.$$
By Lemma \ref{Lem:3a}, we also get a functor 
$$F:\! _B\mathcal{M}\rightarrow\!  _{\mathcal{P}A}\mathcal{M}^H,\quad M\mapsto A\ot_BM.$$

\begin{proposition}\label{pro:3f}
Let $M\in\! _{\mathcal{P}A}\mathcal{M}^H$ and $N\in\! _B\mathcal{M}$. There is a functorial isomorphism of transposed Poisson $(A, H)$-Hopf modules
\begin{align*}
\psi:\! _{\mathcal{P}A}Hom^H(A\ot_B&N,M)\rightarrow Hom_B(N,M^{AcoH})\\
&f\mapsto\psi(f):N\rightarrow M^{AcoH},\ n\mapsto f(1\ot_B n).
\end{align*}
Thus the functors $F$ and $G$ form an adjoint pair with unit and counit
\begin{align*}
&\eta_N:N\rightarrow (A\ot_BN)^{AcoH},\ n\mapsto 1\ot_B n,\\
&\epsilon_M:A\ot_B M^{AcoH}\rightarrow M,\ a\ot_B m\mapsto a\cdot m.
\end{align*}
\end{proposition}

\begin{proof}
For all $a,a'\in A,n\in N$,
\begin{align*}
\{a,a'\}\cdot f(1_A\ot_Bn)&=f(\{a,a'\}\cdot(1_A\ot_Bn))=f(\{a,a'\}\ot_Bn)\\
&=f(a\di(a'\ot_Bn)=a\di f(a'\ot_Bn)\\
&=a\di(a'\cdot f(1_A\ot_Bn)),
\end{align*}
hence $f(1_A\ot_Bn)\in M^{A}$. Since $f$ is a morphism of $H$-comodule, we could obtain that $f(1_A\ot_Bn)\in M^{coH}$. Thus $f(1_A\ot_Bn)\in M^{AcoH}$. Also it is straightforward to verify that $\psi(f)$ is an $B$-module map.  Therefore $\psi$ is well defined. 

Now define
\begin{align*}
\psi':Hom_B(N,&M^{AcoH})\rightarrow\! _{\mathcal{P}A}Hom^H(A\ot_BN,M)\\
&g\mapsto\psi'(g):A\ot_B N\rightarrow M,\ a\ot_B n\mapsto a\cdot g(n).
\end{align*}
For any $a,a'\in A,n\in N$, we have
\begin{align*}
&\psi'(g)(a'\di(a\ot_B n))=\psi'(g)(\{a',a\}\ot_B n)\\
&=\{a',a\}\cdot g(n)=a'\di(a\cdot g(n))\\
&=a'\di\psi'(g)(a\ot_B n),
\end{align*}
where the third identity holds since $g(n)\in M^{AcoH}$, that is, $\psi'(g)$ is Lie $A$-linear. Easy to see that $\psi'(g)$ is $A$-linear and $H$-colinear. Thus $\psi'$ is also well defined. It is a routine exercise to check that $\psi\circ\psi'$ and $\psi'\circ\psi$ are respectively the identity of $Hom_B(N,M^{AcoH})$ and $_{\mathcal{P}A}Hom^H(A\ot_BN,M)$. The proof is completed.
\end{proof}

\begin{corollary}
Let $M$ be a transposed Poisson $(A,H)$-Hopf module. Assume that  there exists a homomorphism of right $H$-comodule $\phi:H\rightarrow A^A$ which is also an algebra map that the identities (\ref{3a}) and (\ref{3b}) hold. Then the functor $G=(-)^{AcoH}:\!  _{\mathcal{P}A}\mathcal{M}^H\rightarrow \! _B\mathcal{M}$ is dual Maschke, that is, every object of $_{\mathcal{P}A}\mathcal{M}^H$ is $G$-relative projective.
\end{corollary}

\begin{proof}
The result follows from Theorem \ref{Thm:3d}, Proposition \ref{pro:3f} and \cite[Theorem 3.4]{CM}.
\end{proof}

\begin{example}
Let $G$ be an affine algebraic group and $A$ a rational $G$-module algebra. From \cite[Example 2.19]{Gu}, we know that a vector space $M$ is a rational $(A, G)$-module if it is an $A$-module, a rational $G$-module, and 
$$g(am)=(g\cdot a)(gm),\quad \forall a\in A,m\in M,g\in G.$$

Let $\bk[G]$ be the affine coordinate ring of $G$. It is well known that a rational $G$-module are $\bk[G]$-module. 

 A rational $G$-module transposed Poisson algebra is a transposed Poisson algebra $A$ which is also a commutative rational $G$-module algebra such that
 $$g\cdot\{a,a'\}=\{g\cdot a,g\cdot a'\},\quad a,a'\in A,g\in G.$$
 Then $A$ is a right $\bk[G]$-comodule transposed Poisson algebra.
 
 The category of rational transposed $(A,G)$-modules and the category of transposed $(A,\bk[G]$-Hopf modules are equivalent. A rational transposed Poisson $(A,G)$-module is a rational transposed $(A,G)$-module $M$ which is also a Poisson $A$-module such that
$$g(a \di m) = (g\cdot a) \di (g m), \quad a \in A, g \in G, m \in M.$$
We denote by ${}_{TP,A,G}\mathcal{M}$ the category of rational transposed Poisson $(A,G)$-modules: its morphisms are the Poisson $A$-linear maps which are also $G$-linear. One could prove that the category of rational transposed Poisson $(A,G)$-modules and the category of transposed Poisson $(A,\bk[G])$-Hopf modules are equivalent.
\end{example}

\vspace{0.1cm}
 \noindent
{\bf Acknowledgements. } 

This work was supported by the NNSF of China (Nos. 12271292, 11901240).

\bibliographystyle{amsplain}

\end{document}